\documentclass[12pt]{article}

\usepackage[latin1]{inputenc}                  
\usepackage[T1]{fontenc}
\usepackage{amsmath,amsfonts,amssymb}
\usepackage{latexsym}

\newcommand{\Ac}{\mathcal A}
\newcommand{\Bc}{\mathcal B}
\newcommand{\Vc}{\mathcal V}
\newcommand{\Mc}{\mathcal M}
\newcommand{\Sc}{\mathcal S}
\newcommand{\Zc}{\mathcal Z}
\newcommand{\Nc}{\mathcal N}
\newcommand{\Pc}{\mathcal P}

\newcommand{\Lc}{\mathcal L}
\newcommand{\Jc}{\mathcal J}
\newcommand{\Tc}{\mathcal T}

\newcommand{\Rb}{\mathbb{R}}
\newcommand{\Cb}{\mathbb{C}}
\newcommand{\Hb}{\mathbb{H}} 
\newcommand{\Nb}{\mathbb{N}}

\newcommand{\Sb}{\mathbb{S}}

\newcommand{\nn}[1]{| #1 |}
\newcommand{\nd}[1]{\left\|#1 \right\|}

\newcommand{\pr}[2]{\mbox{pr}_{#1}(#2)}

\newcommand{\fbs}[1]{\Jc_{#1}}
\newcommand{\fb}[2]{\fbs{#1}(#2)}

\newcommand{\flsz}[1]{\Lc_{#1}}
\newcommand{\flz}[2]{\flsz {#1} ({#2})}

\newcommand{\ssm}{M_v^{ss}}

\newcommand{\supp}{\mbox{\rm supp } }

\usepackage{theorem}

\newtheorem{thm}{Theorem}

\newtheorem{lem}[thm]{Lemma}
\newtheorem{prop}[thm]{Proposition}
\newtheorem{cor}[thm]{Corollary}

\newenvironment{proof}[1][]
{\begin{trivlist}
  \item[]\hspace{\parindent}{\em
      Proof {#1}:}}
  {\hfill $\square$
  \end{trivlist}}

\title{The Spherical Maximal Function 
on the Free Two-step Nilpotent Lie Group}

\author{V\'eronique Fischer
\footnote{Author partially supported by the European Commission (IHP Network Harp).}}

\date{\today}

\begin{document}

\maketitle


\begin{abstract}
We consider here  the free two step nilpotent Lie group,
provided with the homogeneous Kor\'anyi norm;
we prove the $L^p$-boundedness of the maximal function 
corresponding to the homogeneous unit sphere,
for some $p$.
\end{abstract}

\paragraph{Key words.}
spherical means, maximal functions, nilpotent Lie groups, 
spherical function, Plancherel formula.


\section{Introduction}

This article is organized as follows.
As introduction,
we collect some results and methods of work
on singular maximal functions
corresponding to averages 
over dilations of a fixed surface
on a nilpotent Lie group.
We also present our main result.

Then we define precisely the free two-step nilpotent Lie groups, 
maximal $L^p$-inequality and the square functions used here
(Section~\ref{sec_def_not}).

Next (Section~\ref{sec_proofLp}), we prove 
$L^p$-maximal inequalities
corresponding to Kor\'anyi homogeneous spheres,
assuming $L^2$-boundedness of our square functions.
This hypothesis is proved in Section~\ref{sec_squarefcn},
with the help of the Plancherel formula
and spherical functions;
these notions are recalled in Section~\ref{sec_not_tool}.

\subsection{Some maximal functions
  on nilpotent Lie groups}
\label{subsection_max_fcn_nilp}

\paragraph{On the Euclidean space.} 
Stein proved
$L^p$-maximal inequalities 
for Euclidean spheres; 
that is the supremum of the averages
over the spheres centered at a point 
defines an operator on the space of smooth functions, 
called ``the maximal function 
corresponding to Euclidean spheres'' 
(or more simply ``the spherical maximal function'');
and this operator extends to a bounded operator on $L^p(\Rb^n)$.
Here $p>n/(n-1)$, $n\geq 3$ (sharp result).
Bourgain showed the case $n=2$ \cite{bourgain};
the case $n=1$ is vacuously true.

For the proof  in \cite{stein_Maxfunc,stein_wainger},
Stein uses:
\begin{itemize}
\item[(i)] 
 complex interpolation for the analytic continuation 
  of the family of convolution operator with kernel
  given by
$x\mapsto
2{(1-\nn{x}^2)}_+^{\alpha-1}/\Gamma(\alpha)$,
\item[(ii)] 
 the well known $L^p$-boundedness 
  of the standard maximal function
  (corresponding to the Euclidean balls)
  \cite{coifman_weiss},
\item[(iii)] 
  the $L^2$-boundedness of some square functions
(with the help of the Euclidean Fourier transform
  and estimates for special functions).
\end{itemize}

But Stein gave another proof of the same result
in Corollary 3 of \cite[XI\S 3] {stein_Han},
where he also presented the case of variable manifolds.
The method is based there on 
oscilatory integrals operators and on curvature:
in fact, the curvature of the unit Euclidean sphere~$\Sb^n$
 is closely related to the Fourier transform of
the probability measure of~$\Sb^n$
\cite[XI\S 1] {stein_Han}.

\paragraph{On Heisenberg groups.}
M. Cowling offered \cite{cowling} an alternative approach 
to Stein's first work on $\Rb^n$;
he adapted it on the Heisenberg group $\Hb^n$
for the Kor\'anyi norm:
he obtained the $L^p(\Hb^n)$-boundedness of
 the spherical maximal function, $p>(2n+1)/2n$ (sharp result).

Still on the Heisenberg group $\Hb^n=\Cb^n\times \Rb$,
Nevo and Thangavelu considered 
the case of the complex spheres
$\{(z,0)\in\Cb^n\times \Rb : \nn{z}=r\}$,
and showed \cite[\S3.6]{ThangB}
that the corresponding maximal function 
is $L^p$-bounded for $p>(2n-1)/(2n-2)$.
He adapted previous works of Nevo
on ``ergodic'' maximal  inequalities
on some Lie groups 
(see \cite{margulis_nevo_stein_semisimplegr,Nevo}),
where the proofs are based (for the global part) 
on classical study 
of maximal functions for semi-groups
\cite[III\S 3]{stein_topics_in_an},
and on $L^2$-estimates of square functions 
with spectral calculus (spherical functions).\\
The optimal range for $p$ is $p>2n/(2n-1)$;
this was proved recently by Narayanan and Thangavelu \cite{thangavelu2004}.

\paragraph{On nilpotent Lie groups.}
M\"uller and Seeger obtained \cite{muller_seeger}
also this optimal range for $p$ in a study
which concerns more general surfaces 
(with a non vanishing rotational curvature)
in ``non-degenerate''
two-step nilpotent Lie groups. 

In a still unpublished work \cite{schmidt},
the maximal functions on stratified group
associated to hyper-surfaces 
with a non vanishing rotational curvature,
were also studied  
by adapting Stein's second proof of the Euclidean case
\cite[Corollary 3 in XI\S 3] {stein_Han};
it yields to their $L^p$-boundedness
with $p>n/(n-1)$, 
where $n\geq 3$ denotes the topological dimension of the group.
In particular,
it covers the case of the Kor\'anyi spheres
in groups of Heisenberg type,
with a sharp range for $p$, 
and so includes the result of Cowling \cite{cowling} on $\Hb^n$.

\subsection{Our main result} 

Let $N_v$ be the free two step nilpotent Lie group
with $v$ generators, provided as a stratified group, 
with the Kor\'anyi norm.
Let $\Ac$ denote the spherical maximal functions.
We recall the definitions of these objects 
in Section~\ref{sec_def_not}.

The main result of this paper is the $L^p$-boundedness of~$\Ac$:
\begin{thm}
  \label{mainthm}
Assume $v\geq 4$. The spherical maximal function $\Ac$ satisfies an $L^p(N_v)$-inequality
where $2h/(2h-1)<p\leq\infty$
 and   $h$ is the largest integer such that $h<v(v-1)/4-1$.
\end{thm}

This result is new, but probably not sharp for the range of $p$.
The group $N_v$ does not fulfill the non-degeneracy condition 
of \cite{muller_seeger},
and the surface is not of the type studied there neither.
It is not covered by the study~\cite{schmidt}
because the curvature of the unit homogeneous sphere in $N_v$
 vanishes ``at the equator'',
$v\geq 3$.

Our proof follows the first method of the Euclidean case 
\cite{stein_Maxfunc},
i.e. the points (i), (ii), (iii);
but the Euclidean norm is replaced by the Kor\'anyi norm
and for $L^2$-estimates of square functions
(i.e. for (iii)), 
we use the Plancherel formula and the bounded spherical functions.
This method for groups of Heisenberg type
(and always Kor\'anyi homogeneous spheres)
gives a result included in~\cite{schmidt}
(see the French PhD thesis of the  author \cite{moi}).

The rest of this paper is organized as follows.
We start by setting definitions and notations 
for the group, the spherical maximal function
and square functions denoted $S^j$
(Section~\ref{sec_def_not}).
We give then the proof of our main theorem
(Section~\ref{sec_proofLp}),
provided that some square functions $S^j$  
satisfy $L^2$-estimates.
Next (Section~\ref{sec_squarefcn})
we prove this hypothesis,
using the Plancherel formula and spherical functions 
(notions recalled in Section~\ref{sec_not_tool}).


\section{Definitions and Notations}
\label{sec_def_not}

Here we set definitions and notations 
for the group, the spherical maximal function
and the square functions that we study.

In this paper, we assume $v$ is an integer such that $v\geq 2$.

\subsection{Free two-step nilpotent Lie algebras and groups}
\label{subsec_def_not_lie_alg_gr}

Let $\Nc_v$ be the 
(unique up to isomorphism)
{\it free two-step nilpotent Lie algebra} with $v$ generators;
the definition using the universal property 
can be found in \cite[Chapter V \S 4]{jacobson}.
Roughly speaking, $\Nc_v$ is a two-step nilpotent Lie algebra
with $v$ generators $X_1,\ldots,X_v$,
such that the vectors  $X_1,\ldots,X_v$ and
$X_{i,j}=[X_i,X_j]$, $i<j$, form a basis of the vector space $\Nc_v$.

We denote by
$\Vc$ and $\Zc$,
the vector spaces generated by 
the families of vectors $X_1,\ldots,X_v$
and $X_{i,j}:=[X_i,X_j], 1\leq i< j\leq v$
respectively;
these families become the canonical bases of $\Vc$ and $\Zc$.
Thus $\Nc_v=\Vc\oplus\Zc$ is a stratified algebra (see \cite{folland_stein}),
and $\Zc$ is the center of the Lie algebra $\Nc_v$. 
Let $z=v(v-1)/2$ be the dimension of $ \Zc$
and $Q=v+2z=v^2$ be the homogeneous dimension.
We write $v=2v'$ or $2v'+1$.

The connected simply connected nilpotent Lie group
which corresponds to $\Nc_v$ is called the
{\it free two-step nilpotent Lie group} and is denoted $N_v$.
We denote by $\exp :\Nc_v\rightarrow N_v$ the exponential map.
In this paper, we use the notations $X+A\in\Nc$, $\exp(X+A)\in N$
when $X\in \Vc,A\in\Zc$.
For an element $n=\exp(X+A)\in N$,
we define the dilations: 
$r.n=\exp(rX+r^2A)$, $r>0$,
and the (homogeneous) Kor\'anyi norm: 
$\nn{n}={( \nn{X}^4+\nn{A}^2 )}^{\frac14}$ (see \cite{folland_stein}).

\subsection{Some measures}

The canonical bases of $\Vc$ and $\Zc$
induce Lebesgue measures on  $\Vc$ and $\Zc$ respectively,
thus on $\Nc$, and a Haar measure $dn$ on $N_v$.
In this paper, 
the $L^p$ spaces over $N_v$
and the measure $\nn{E}$ for a Borel set $E\subset N_v$
are related to the Haar measure~$dn$.

We denote $\mu$
the unique Radon measure
on the unit homogeneous sphere 
$S_1:=\{n\in N_v\, ,\, \nn{n}=1\}$ 
such that 
for each integrable function $f$ on $N_v$, 
we have
\cite[prop.1.15]{folland_stein}:
\begin{equation}
  \label{formula_polar_coordinate}
  \int_{N} f(n)dn    
  \,=\,
  \int_{r=0}^\infty\int_{S_1} f(r.n)d\mu(n)r^{Q-1}dr 
  \quad.
\end{equation}
By uniqueness of (\ref{formula_polar_coordinate}),
and with suitable changes of variables, 
we easily obtain:
\begin{eqnarray}
  \int_{S_1} f(n)d\mu(n)  
  \,=\,
  2\int_{0}^1 \int_{\Sb^v}\int_{\Sb^z}
  f(\exp t X +\sqrt{(1-t^4)}Z)\nonumber\\
  d\sigma_z(Z)
  d\sigma_v(X)
  t^{v-1}{(1-t^4)}^\frac{z-2}2dt
  \quad .\label{expression_mu}
\end{eqnarray}
(where $\sigma_n$  denotes
the probability measure on the Euclidean sphere
$\Sb^n\subset\Rb^n$).
It is still valid for any two-step stratified group.

\subsection{Maximal $L^p$-inequality}
\label{subsec_maxin}

Here, we recall the definition
of the maximal function associated to an operator over $N_v$,
in particular of the spherical maximal function $\Ac$,
and of the standard maximal function $\Mc$; 
this is still valid for any stratified group.

Let us start with the definitions of 
{\it dilations for subsets, distributions and operators} over $N_v$.
Let $r>0$.
The $r$-dilation of $E\subset N_v$ is $r.E=\{r.n, n\in E\}$,
and if $E$ is a measurable set, 
we have $\nn{r.E}=r^Q\nn{E}$.
The dilation $F_r$ of a distribution $F$ on $N_v$,
is  given for a test function~$f$ on~$N_v$ by
$<F_r,f>=<F,f(r.)>$
where
$f(r.):n\mapsto f(r.n)$.
The $r$-dilation $T_r$ of an operator $T$
 on a dilation invariant space of function on $N_v$ is given 
for a function $f$ by:
$T_r.f=
\left(T\,f(r.)\right)(r^{-1}.)$.
For a convolution operator
with an integrable function or a measure,
the $r$-dilated operator coincides with the convolution operator 
with the corresponding $r$-dilated function or measure.

Let  $T$  be an operator on a dilation invariant space of function on $N_v$.
We define the maximal function associated to $T$:
$\Tc:f \mapsto \sup_{r>0} \nn{T_r.f}$.
We say that 
{\it $\Tc$ satisfies an $L^p$-inequality} 
(or $T$ satisfies an $L^p$ maximal inequality,
or $\Tc$ is $L^p$-bounded)
if there exists a constant
$C=C(T,v,p)>0$ such that for all smooth compactly supported functions 
 $f$ on $N_v$ 
(or equivalently 
the class $L^p(N_v)$, the Schwartz class $\Sc(N_v)$ \ldots),
we have:
$\nd{\Tc.f}_{L^p}
\leq 
C\nd{f}_{L^p}$.

The {\it maximal spherical function} $\Ac$
is the maximal function associated to the convolution operator with 
the measure $\mu$ (see~(\ref{formula_polar_coordinate})).
The aim of this paper is to establish $L^p$-inequalities for $\Ac$.

Let $\Mc$  denote
 the {\it standard maximal function},
i.e. the maximal function associated to the characteristic function 
of the unit homogeneous ball $B_1:=\{n\in N_v;\nn{n}\leq 1\}$.
During the proof of our main result,
we will use the well known 
$L^p$-inequalities ($1<p\leq\infty$ ),
satisfied by the standard maximal function $\Mc$
\cite{coifman_weiss}.
We will also need the following corollary:
\begin{cor}
  \label{cor_decfcn}
Let $m:\Rb^+\rightarrow \Rb^+$ be a decreasing function
and $F:N\rightarrow \Rb^+$ be defined by $F(n)=m(\nn{n})$, $n\in N_v$.
Suppose $F\in L^1$.
The maximal function 
associated to the convolution operator with $F$
satisfies $L^p$-inequalities, 
for $1<p\leq \infty$.
\end{cor}

\subsection{Square functions}

For $j=1,2,\ldots$,
we define the following operator, called the square function $S^j$, by:
$$
S^j(f)(n) 
\,:=\,
\sqrt{ \int_{s=0}^\infty  \nn{\partial_s^j  (f*\mu_s)(n) }^2
  s^{2j-1} ds }
\quad,\quad f\in \Sc(N_v)
\quad,\quad n\in N_v\quad.
$$
We will have to show that for some $j$,
the square functions $S^j$ satisfy $L^2$-estimates:
$$
\mbox{i.e}\quad \exists C=C(v,j)>0 \quad
\forall f\in \Sc(N_v)
\qquad
\nd{S^j.f}_{L^2}
\,\leq\, C
\nd{f}_{L^2}
\; .
$$
In Section~\ref{sec_squarefcn},
using the Plancherel formula and spherical functions,
we will obtain:
\begin{thm}
  \label{thm_L2_S}
We assume $v\geq 4$.
The square functions $S^j$, $j=1,\ldots,h$,
   satisfy $L^2$-estimates, 
where $h\in \Nb$, $1\leq h < (z-2)/2$.
\end{thm}


\section{Proof of $L^p$-maximal estimates for $\Ac$ }
\label{sec_proofLp}

Here we will prove that the $L^2$-boundedness of some square functions
$S^j$ implies the $L^p$-boundedness of $\Ac$, for some $p$.
We summarize this fact in the following Theorem:
\begin{thm}
\label{thm_Lp_Ac}
If the square functions $S^j$ satisfy $L^2$-estimates,
then the spherical maximal function $\Ac$ satisfies 
$L^p$-estimates, where:
\begin{itemize}
\item[a)] $j=1$ and $2\leq p\leq\infty$ if $v\geq 2$, 
\item[b)] $1<j\leq h <(Q-2)/2$
and $2h/(2h-1)<p\leq \infty$  if $v>2$.
\end{itemize}
\end{thm}
This proof extends naturally to any stratified Lie groups.
Furthermore Theorems~\ref{thm_L2_S} and~\ref{thm_Lp_Ac}
imply Theorem~\ref{mainthm}.
In this section, we denote $N_v=N$.

\subsection{Proof of the $L^p$-case, $p\geq 2$}
\label{subsec_proofL2}

Here we prove Theorem~\ref{thm_Lp_Ac}.a):
we assume that the square function $S^1$ is $L^2$-bounded
and we prove the $L^p$-boundedness of $\Ac$, $p\geq 2$.

We follow the Euclidean process \cite[XI\S 1]{stein_Han}.
For $f\in \Sc(N)$, we have:  
\begin{eqnarray}
  \mu_t*f(n)
  &=&
  \frac1{t^Q}\int_{0}^t\partial_s(s^Q\mu_s*f(n)) ds
  \nonumber\\  
  &=&
  \frac1{t^Q}\int_{0}^tQs^{Q-1}\mu_s*f(n)ds
  +\frac1{t^Q}\int_{0}^ts^Q\partial_s(\mu_s*f(n)) ds
  \quad .\label{rhs_mu_t*f}
\end{eqnarray}
We compute
$t^Q={\nn{B(n,t)}}/\nn{B_1}$,
where $B(n,r)=\{n'\in N\,,\, \nn{n{n'}^{-1}}\leq r\}$ 
is the homogeneous ball with radius~$r$, centered at~$n$,
and with change in polar coordinates~(\ref{formula_polar_coordinate}):
$$
\frac1{t^Q}\int_{0}^tQs^{Q-1}\mu_s*f(n)ds
\,=\,
Q\frac{\nn{B_1}}{\nn{B(n,t)}}
\int_{B(n,t)}
f(n'') dn''
\quad .
$$
Because of H\"older's inequality, 
the last expression in (\ref{rhs_mu_t*f}),
is majorized by
${(2Q)}^{-1/2} S^1.f(n)$.
Finally, we gather:
$$
\nn{\mu_t*.f(n)} 
\,\leq\, 
Q\frac{\nn{B_1}}{\nn{B(n,t)}}
\int_{B(n,t)}
f(n'') dn''
+\frac1{\sqrt{2Q}}S^1.f(n)
\quad.
$$
Taking the supremum over $t$,
we have the following pointwise estimate:
\begin{equation}
  \label{eq_Ac_Mc_s}
\forall f\in\Sc(N)
\qquad
\Ac.f 
\,\leq\, 
Q\nn{B_1}
\Mc.f +\frac1{\sqrt{2Q}}S^1.f 
\quad .
\end{equation}
The standard maximal function $\Mc$ satisfies
$L^p$-inequalities ,
in particular for $p=2$ 
\cite{coifman_weiss}.
As the square function $S^1$ is assumed to be $L^2$-controlled,
the inequality~(\ref{eq_Ac_Mc_s}) leads to an $L^2$-inequality for~$\Ac$.
But $\Ac$ satisfies also an $L^\infty$-inequality
(because the measure $\mu$ is finite).
By classical real interpolation,
we obtain the $L^p$-boundedness of~$\Ac$,
for $p\geq 2$.
Theorem~\ref{thm_Lp_Ac}.a) is thus proved.

\subsection{Overview of the proof of the $L^p$-cases, $p<2$}
\label{subsec_proofLp}

Here we prove  Theorem~\ref{thm_Lp_Ac}.b).
We start with a choice of an analytic family of operators~$A^\alpha$,
the same as in the Euclidean case
\cite{stein_Maxfunc,stein_wainger}, 
see (i) of Subsection~\ref{subsection_max_fcn_nilp}.
$A^0$ coincides with the convolution operator with kernel $\mu$.
Provided that the square functions $S^j$, $1\leq j\leq h$,
satisfy $L^2$-estimates,
we prove that $A^\alpha$ satisfies $L^p$-maximal inequalities,
for $\Re\alpha\geq 1$ with $1<p\leq\infty$ 
(Subsection~\ref{subsec_Lp_max_ineq_xgeq1}),
and for $\Re\alpha>-h+1$ with $p=2$
(Subsection~\ref{subsec_Lp_max_ineq_x<0}).
To obtain these maximal inequalities, we will use 
Corollary~\ref{cor_decfcn}, and the following estimate
``uniformly in $y\in \Rb$, locally in $x>0$''
for the $\Gamma$ function  \cite[page 151]{tit}:
\begin{eqnarray}
&\forall  [a,b]\subset [0,\infty[
    \quad
    \exists C>0
    \quad
\forall x\in [a,b]
\quad 
    \forall y\in \Rb \nonumber&\\
&C^{-1} \leq 
  \frac{e^{-\frac{\pi}{2}y}{\nn{y}}^{x-\frac{1}{2}}}
{\nn{\Gamma(x+iy)}}
\leq C  \; .&  
  \label{maj_tit}
\end{eqnarray}
From these maximal inequalities, 
with interpolation 
(Subsection~\ref{subsec_interpolation}), 
we will deduce  Theorem~\ref{thm_Lp_Ac}.b).

\subsection{Analytic family of operators}

Let define (as in the Euclidean case,
see \cite{stein_Maxfunc} or Subsection~\ref{subsection_max_fcn_nilp})
for $\alpha\in \Cb\backslash(-\Nb)$:
$$
m^\alpha(r)
=
\frac{2{(1-\nn{r}^2)}_+^{\alpha-1}}{\Gamma(\alpha)}
\quad,\quad r\geq 0\quad,
$$
and
$$
\mu^\alpha(n)
=
m^\alpha(\nn{n})
\quad,\quad n\in N\quad.
$$
For $\Re\alpha>0$, 
the function $\mu^\alpha$ is radial and integrable;
 we denote by $A^\alpha$ the convolution operator 
with kernel $\mu^\alpha$;
it is a $L^p$-bounded operator for $1<p\leq \infty $.

Let $j,h\in\Nb$ 
such that $0\leq j\leq h$
and $1\leq h >Q/2$;
we define also
the operator $B^\alpha_{h,j}$ for $f\in\Sc(N)$,
and $n\in N$, by:
$$
 B^\alpha_{h,j}(f)(n) = \int_{r=0}^1
  m^{\alpha+h}(r) r^{Q-1-2h+j} \partial_r^j (f*\mu_r)(n) dr
  \quad .
  $$ 
\begin{prop}
\label{prop_analytic_cont}
The family  of operators $\{ A^\alpha, \Re \alpha >0\}$
is analytic, and admits an analytic continuation
over $\{\Re \alpha >1-(Q-2)/2\}$
as operators on $\Sc(N)$.
The operator $A^0$ coincides with the convolution operator 
with kernel $\mu$.
For $\Re \alpha >1-h$ with $h< (Q-2)/2$,
$A^\alpha$ can be written as linear combination
of the operators $B^\alpha_{h,j}$, $0\leq j\leq h$.
\end{prop}

\begin{proof}[of Proposition~\ref{prop_analytic_cont}]
With the change in polar coordinates~(\ref{formula_polar_coordinate}):
\begin{eqnarray*}
  A^\alpha.f
&=&
  \int_{r=0}^1 
  \frac2{\Gamma(\alpha)}{(1-r^2)}^{\alpha-1}  
  (f*\mu_r) r^{Q-1} dr\\
&=&
  \left[  \frac2{\alpha\Gamma(\alpha)} {(1-r^2)}^{\alpha}  
    \, 
    \frac{1}{-2r}(f*\mu_r) r^{Q-1} \right]_{r=0}^1\\
&&\quad  -  \int_{r=0}^1 
  \frac2{\alpha\Gamma(\alpha)} {(1-r^2)}^{\alpha}
  \,
  \partial_r \left( \frac{1}{-2r}  (f*\mu_r) r^{Q-1} \right)
  dr 
  \; ,
\end{eqnarray*}
after integrating by parts.
The endpoint terms equal zero, 
and $\alpha\Gamma(\alpha)=\Gamma(\alpha+1)$; 
we easily obtain 
on one hand  for $\alpha=0$:
$A^0.f=f*\mu$,
and on the other hand after expanding the derivative:
$A^\alpha=-1/2(B_{1,1}+QB_{1,0})$.

Recursively,  using $h$ integrations by parts,
we easily compute that the operator $A^\alpha$
coincides with a linear combination
of the operators $B^\alpha_{h,j}$, 
provided that the endpoint terms equal zero,
so as long as $h<(Q-2)/2$ and $\Re \alpha >1-h$.
\end{proof}

Now we establish some $L^p$-maximal inequalities on $A^\alpha$.
We denote by $\Ac^\alpha$ the maximal function associated to $A^\alpha$
(see Subsection~\ref{subsec_maxin}).

\subsection{$L^p$-maximal inequalities for $\Re\alpha \geq 1$}
\label{subsec_Lp_max_ineq_xgeq1}

In this subsection, 
using  Corollary~\ref{cor_decfcn} and estimation~(\ref{maj_tit}),
we prove: 
\begin{prop}
\label{prop_Lp_ineq}
  For $1<p\leq \infty$, 
we have:
\begin{eqnarray*}
&\forall b\geq 1\quad
\forall x\in [1,b] \quad 
    \exists C>0 \quad
    \forall y\in \Rb \quad
  \forall f\in \Sc(N),&\\
&  \nd{\Ac^{x+iy}.f}_{L^p} 
  \leq \,C\, e^{2 \nn{y}} \nd{f}_{L^p}
  \;.&  
\end{eqnarray*}
\end{prop}

\begin{proof}[of Proposition~\ref{prop_Lp_ineq}]
For $x>0$, and $ f\in \Sc(N)$,
we have:
$$
  \nn{A^{x+iy}.f}
  \leq
  \nn{f}*\nn{\mu^{x+iy}} 
=
\nn{\frac{\Gamma(x)}{\Gamma(x+iy)}} 
\nn{f}*\mu^x
  \leq
C e^{2  \nn{y}} \,A^x.\nn{f}
    \; .
$$
uniformly in $y\in \Rb$, locally for $x>0$,
because of~(\ref{maj_tit}).
It yields to 
$\Ac^{x+iy}.f\leq C e^{2  \nn{y}} \Ac^x.\nn{f}$.
Now, For $x\geq 1$, $\mu^x$ satisfies the hypothesis 
of Corollary~\ref{cor_decfcn}, 
and the maximal function $\Ac^x$ is $L^p$-bounded, 
$1<p\leq\infty$. 
\end{proof}

\subsection{$L^p$-maximal inequalities for $\Re\alpha<0$}
\label{subsec_Lp_max_ineq_x<0}

In this subsection, 
we fix $h=1,2,\ldots$
and we prove: 
\begin{prop}
\label{prop_L2_ineq}
If the square functions $S^j$, $j=1,\ldots,h$
satisfy $L^2$-estimates,
then for all segments $[a,b]\subset]-h+1, \infty[$,
there exists a constant $C>0$ such that:
  $$
\forall x\in [a,b] \quad 
    \forall y\in \Rb \quad
  \forall f\in \Sc(N),\quad 
  \nd{\Ac^{x+iy}.f}_{L^2} \leq \,C\, e^{2 \nn{y}} \nd{f}_{L^2}
  \;.
  $$
\end{prop}
This proposition is implied by the following lemma, 
which states estimates 
for the maximal function $\Bc^\alpha_{h,j}$
associated (see Subsection~\ref{subsec_maxin})
to $B^\alpha_{h,j}$:

\begin{lem}
\label{lem_ineq_Bhj}
For all segments $[a,b]\subset]-h+1, \infty[$,
there exists a constant $C>0$ such that 
for all $x\in [a,b]$
and for all $y\in \Rb$,
we have:
\begin{eqnarray}
  \label{eq_lem_ineq_Bhj_j}
  &\forall\, f\in L^2&
  \quad
  \nd{\Bc^{x+iy}_{h,0} .f}_{L^2} 
  \,\leq\,  \;C\; e^{ 2 \nn{y}} \nd{f}_{L^2}
  \; ,\\
  \label{eq_lem_ineq_Bhj_0}
  &\forall\, f\in \Sc(N)&
  \quad
  \nn{\Bc^{x+iy}_{h,j} .f}
  \,\leq\,  \;C\; e^{2 \nn{y}} \nn{S^j.f}
  \; .
  \end{eqnarray}
\end{lem}

The proof of Lemma~\ref{lem_ineq_Bhj} is based on 
 Corollary~\ref{cor_decfcn}, known formulae for $\Gamma$
and estimation~(\ref{maj_tit}).

\begin{proof}[of lemma~\ref{lem_ineq_Bhj}]
Because of the change in  polar coordinates~(\ref{formula_polar_coordinate}), 
we compute that
 $B^\alpha_{h,0}$ is the convolution operator with kernel
given by:
$n\mapsto\mu^{\alpha+h}(n)\nn{n}^{-2h}$,
$n\in N$.
Then we proceed as in the proof of Proposition~\ref{prop_Lp_ineq}
to obtain inequality~(\ref{eq_lem_ineq_Bhj_j}).

Let us prove inequality~(\ref{eq_lem_ineq_Bhj_0}).
For $f\in \Sc(N)$, we have:
$$
  {(B^\alpha_{h,j})}_t.f 
  \,=\, 
 \int_{r=0}^1 m^{\alpha+h}(r)
  r^{Q-1-2h+j} t^j [\partial_{r'}^j f*\mu_{r'}]_{|r'=rt} dr 
\quad.
$$
Thus because of H\"older's inequality,
${(\Bc^\alpha_{h,j}.f)}^2$
is majorized by the product of
\begin{eqnarray*}
\sup_{t>0}\int_{r=0}^1 \nn{ r^{j-\frac12}  t^j  
          [\partial_{r'}^j f*\mu_{r'}]_{|r'=rt} }^2dr  
&\leq&
\int_{r'=0}^\infty      
      \nn{ \partial_{r'}^j f*\mu_{r'}}^2 {r'}^{2j-1} dr'\\
&&\quad=\,
      {\left(S^j(f)\right)}^2\quad,
\end{eqnarray*}
(with the help of change of variable  $r'=rt$)
and the integral:
\begin{eqnarray*}
\int_{r=0}^1  \nn{ m^{\alpha+h}(r) r^{Q-1-2h-\frac12}}^2  dr 
\,\leq\,
  \int_{r=0}^1  \frac{{(1-r^2)}^{2(x+h-1)}}  {\nn{\Gamma(\alpha+h)}^2}
  r^{2Q-1-4h} dr 
\quad.\\
\end{eqnarray*}
This last integral equals with the change of variable $r'=r^2$,
then with known formulae concerning the $\Gamma$ function 
\cite[\S1.7 (4) and (5)]{szego}:
$$
    \int_{r'=0}^1   
\frac{ {(1-r')}^{2(x+h-1)}}
{\nn{\Gamma(\alpha+h)}^2}
    {r'}^{Q-1-2h}  \frac{dr'}2 
=
    \frac 12
    \frac{\Gamma(2x+2h-1)\Gamma(Q-2h)}
    {\nn{\Gamma(\alpha+h)}^2\Gamma(2x+Q)}
\quad.
$$
Because of estimation~(\ref{maj_tit}), 
this term is bounded up to a constant, 
locally in $x>-h+1$, uniformly in $y\in\Rb$,
 by $e^{2\nn{y}}$.
We obtain inequality~(\ref{eq_lem_ineq_Bhj_0}).
\end{proof}

This ends the proofs of Lemma~\ref{lem_ineq_Bhj}
thus of Proposition~\ref{prop_L2_ineq}.

\subsection{Interpolation}
\label{subsec_interpolation}

The maximal function $\Ac=2\Ac^0$ (Proposition~\ref{prop_analytic_cont})
satisfies an $L^\infty$-inequality because the measure $\mu$ is finite.
We have proved in Propositions~\ref{prop_L2_ineq} and~\ref{prop_Lp_ineq},
that the family $\Ac^\alpha$ satisfies  $L^p$-maximal inequalities.
Using linearisation and interpolation of an analytic family,
we obtain that the maximal function $\Ac^0$ satisfies $L^p$-inequalities
for $(2h)/(2h-1)<p\leq \infty$ as long as 
the square functions $S^j$, $1\leq j\leq h$ 
satisfy $L^2$-estimates
and $1<h<(Q-2)/2$
(so $v\geq 3$).

We have proved  Theorem~\ref{thm_Lp_Ac}.b).
This achieves the proof of  Theorem~\ref{thm_Lp_Ac},
and thus of our main theorem,
 provided that we show Theorem~\ref{thm_L2_S}.
The last section (Section~\ref{sec_squarefcn}) 
is devoted to this,
but we start (Section~\ref{sec_not_tool})
by recalling our tools: 
the Plancherel formula and bounded spherical functions.


\section{The Bounded Spherical Functions and the Plancherel Formula }
\label{sec_not_tool}

Here we present the tools
that we will need in the proof of Theorem~\ref{thm_L2_S}:
the Plancherel formula
and bounded spherical functions 
on $N_v$.
For this results, 
we refer the interested reader to \cite{moi}.

\subsection{Action of \mathversion{bold}{$O(v)$}}
\label{subsec_radial_fcn}

With the canonical basis,
the vector space $\Zc$ can be identified with the vector space of
skew-symmetric $v\times v$-matrices $\ssm$.

The group $O(v)$ 
acts on $\Vc_v\sim \Rb^v$ and $\Zc_v\sim\ssm $ by auto-morphisms :
$$
\left\{\begin{array}{rcl}
    O(v)\times \Vc_v
    &\longrightarrow& \Vc_v\\
    (k,X)&\longmapsto&k.X
  \end{array}\right.
\quad\mbox{and}\quad
\left\{\begin{array}{rcl}
    O(v)\times \Zc_v
    &\longrightarrow& \Zc_v\\
    (k,A)&\longmapsto&k.A=kAk^{-1} 
  \end{array}\right. \quad.
$$
We get easily
$k.[X,X']=[k.X,k.X']$, $k\in O(v)$, $X,X'\in\Vc_v$.
So, $O(v)$ acts (by auto-morphisms) on the Lie algebra $\Nc_v$, 
then on the group $N_v$.

A distribution or a left invariant differential operator on $N_v$ 
is said to be \textit{radial},
if it is $O(v)$-invariant.

\subsection{Special functions}
\label{subsec_specialfcn}

First, let set the notations for ``our'' Hermite, Laguerre and Bessel functions.
We denote:
  \begin{itemize}
\item $h_l,l\in\Nb$ the Hermite-Weber functions
 on $\Rb$ given by \cite[\S 5.5]{szego}:
$$
h_l(x)
\,=\,
{(2^l l!\sqrt{\pi})}^{-\frac l2} e^{-\frac{x^2}2} H_l(x)
\quad\mbox{where}\quad
H_l(s)={(-1)}^l e^{s^2}{(d/ds)}^l e^{-s^2}
\; ,$$
\item $\flsz n$, 
the Laguerre function
given by $\flz n x=L_n^{0}(x)e^{-\frac{x}{2}}$
where $L_n^{0}$ is the Laguerre polynomial 
of type $0$ and degree $n$ \cite[\S 5.1]{szego},
\item  $\fbs \alpha$ the reduced Bessel function:
$\fb \alpha{z}:=
\Gamma (\alpha+1){(z/2)}^{-\alpha}J_\alpha(z)$,
where $J_\alpha$ is the Bessel function of type $\alpha>0$
\cite[\S 1.71]{szego}, \cite[ch. II, I.1]{faraut_h}.
\end{itemize}

Now, we can give some spherical  (or equivalently of positive type) functions of $(N_v,O(v))$.
We denote:
\begin{itemize}
\item $dk$ the Haar probability measure of $O(v)$,
\item for $\Lambda=(\lambda_1,\ldots,\lambda_{v'})\in\Rb^{v'}$, 
$D_2(\Lambda)$ an element of $\Zc_v^*$ identified with the skew-symmetric matrix:
$$
D_2(\Lambda)
\,:=\,
\left[
  \begin{array}{cccc}
    \lambda_1 J &&\\
    0 & \ddots & 0&\\
    &&\lambda_{v'} J&\\
    &&&(0)
  \end{array}\right]
\quad \mbox{where}\quad
    J:=
    \left[  \begin{array}{cc}
        0&1\\
        -1&0
      \end{array}\right]\quad.
$$
($(0)$ means that a zero appears only in the case $v=2v'+1$.)
\item $X_1^*,\ldots,X_v^*$ the dual base of $X_1,\ldots,X_v$,
\item $\mbox{pr}_j$ the orthogonal projection on $\Rb X_{2j-1}\oplus\Rb X_{2j}$, 
for $j=1,\ldots,v'$,
\item $\Lc$ the set of
  $\Lambda=(\lambda_1,\ldots,\lambda_{v'})\in\Rb^{v'}$ such that
  $\lambda_1\,>\, \ldots \,>\, \lambda_{v'}\,>\,0$,
\item $\Mc$ 
  the set of
  $(r,\Lambda)$  with $\Lambda\in \Lc$
  and $r>0$ if $v=2v'+1$ or $r=0$ if $v=2 v'$,
\item $\Pc$
  the set of
  $(r,\Lambda,l)$  with $(r,\Lambda)\in \Mc$ and $l\in \Nb^{v'}$.
\end{itemize}

For $(r,\Lambda,l)\in \Pc$,
we define 
the function
$\phi^{r,\Lambda,l}$ by:
$$
\phi^{r,\Lambda,l}(n)
\,=\,
\int_{O(v)}
\Theta^{r,\Lambda,l,\epsilon}(k.n)
dk,
\quad n\in N_v
\quad,
$$
where $\Theta^{r,\Lambda,l}$ is given by:
$$
\Theta^{r,\Lambda,l}(\exp (X+A))
\,=\,
e^{i <rX^*_v,X>}
e^{i<D_2(\Lambda),A>}
\overset{v'}
{\underset{j=1}\Pi}
\flz {l_j} {\frac{\lambda_j}2 \nn{\pr {j} X }^2}
\quad.
$$
The function $\phi^{r,\Lambda,l}$
is a spherical function of $(N_v,O(v))$.
In the rest of this article,
we identify a function
$\phi^{r,\Lambda,l}$,
with its parameter
$(r,\Lambda,l)\in\Pc$.

\subsection{The Plancherel Formula}
\label{subsec_plancherel}

Let $(r,\Lambda)\in\Mc$.
Let define the representation
$(L^2(\Rb^{v'}),\Pi_{r,\Lambda})\in\hat{N}_v$ 
by:
\begin{eqnarray*}
 \Pi_{r,\Lambda}(n).f(y)
\,=\,
\exp \left( i 
\sum_{j=1}^{v'}
  \frac {\lambda_j}2 x_{2j}x_{2j-1} +\sqrt{\lambda_j} x_{2j}y_j\right)
e^{irx_v+i<D_2(\Lambda),A>}\\
f(y_1+\sqrt{\lambda_1}x_1,\ldots,y_{v'}+\sqrt{\lambda_{v'}}x_{2v'-1})
\quad ,
\end{eqnarray*}
where $f\in L^2(\Rb^{v'})$,
$(y_1,\ldots,y_{v'})\in\Rb^{v'}$,
$n=\exp(X+A)\in N_v$ with $X=\sum_{j=1}^v x_jX_j$.

The group $O(v)$ acts on $N_v$ thus also on $\hat{N}_v$. 
In particular, we denote:
$$
k.\Pi_{r,\Lambda}(n)=
\Pi_{r,\Lambda}(k^{-1}.n)
\quad,\quad k\in O(v),\quad n\in N_v\quad.
$$

We denote 
$\eta$
  the measure on $\Lc$ given by:
$$
d\eta(\Lambda)=
\left\{\begin{array}{ll}
  \overset{v'}{\underset{i=1}{\Pi}} \lambda_i
  \underset{j<k}\Pi {(\lambda_j^2-\lambda_k^2)}^2 
  d\lambda_1\ldots d\lambda_{v'}
&
\mbox{if}\; v=2v',\\
  \overset{v'}{\underset{i=1}{\Pi}} \lambda_i^3
  \underset{j<k}\Pi {(\lambda_j^2-\lambda_k^2)}^2 
  d\lambda_1\ldots d\lambda_{v'}
&
\mbox{if}\; v=2v'+1,\\
\end{array}\right.
$$
and by $\tau$ the Lebesgue measure on $\Rb^{*+}$ if $v=2v'+1$,
and the Dirac measure in 0 if $v=2v'$. 
Let $m$ be the measure on $\Mc\times O(v)$
given as the tensor product of $\tau$,
$\eta$ and $dk$, up to a normalizing constant
 so that we have the Plancherel formula:
\begin{thm}
\label{thm_plancherel}
  $m$ is the Plancherel measure of $N_v$:
  $$
\forall f\in\Sc(N_v)\qquad
  \nd{f}_{L^2}^2
  \,=\,
\int
      \nd{k.\Pi_{r,\Lambda}(f) }_{HS}^2
      dm(r,\Lambda,k)
\quad,
  $$
where $\nd{.}_{HS}$ denotes  the norm of an Hilbert-Schmidt operator.
\end{thm}

For $l=(l_1,\ldots,l_{v'})\in\Nb^{v'}$, 
let $\zeta_l\in L^2(\Rb^{v'})$ 
be given by:
$
\zeta_l
(y_1,\ldots,y_{v'})
=\Pi_{j=1}^{v'}
h_{l_j}(y_j)$.
The vectors $\zeta_l$, $l\in\Nb^{v'}$, 
form an orthonormal basis of $L^2(\Rb^{v'})$ .
If $f$ is radial,
the Hilbert-Schmidt norm of~$\Pi(f)$,
 can be computed with this basis, in term of spherical functions:
\begin{prop}
\label{prop_pi(f)_rad}
Let $f$ be a radial integrable function 
(or a radial finite measure) on $N_v$.
We have:
$$
\Pi_{r,\Lambda}(f).\zeta_l
\,=\,
<f,\phi^{r,\Lambda,l}>
\zeta_l
\quad.
$$
\end{prop}


\section{Square functions}
\label{sec_squarefcn}

This section is devoted to the proof of
Theorem~\ref{thm_L2_S}.
Actually, we will prove (Subsection~\ref{subsec_S_hat_S})
that the $L^2$-norm of square functions $S^j$,
$j=1,2,\ldots$,
is bounded by the $L^\infty$-norm of :
$$ 
\hat{S}^j(\phi)
:=
\sqrt{\int_{s=0}^\infty
{\nn{\partial_s^j<\mu_s,\phi>}}^2 s^{2j-1}ds}
\quad, \phi\in\Pc
\quad:
$$
In this section,
 $\nd{.}$ denotes the $L^2(N_v)$-norm.
\begin{prop}
\label{prop_S_hat_S}
We have for $f\in\Sc(N_v)$ and $j\geq 1$:
$$
\nd{S^j(f)}\leq
\sup_{\phi\in\Pc}\hat{S}^j(\phi)
\;\nd{f}
\quad.
$$
\end{prop}

Next (Subsection~\ref{subsec_hatS}), 
we obtain the following boundedness for $\hat{S}^j$:
\begin{prop}
\label{prop_hatS}
Let $h$ be an integer such that 
$1\leq h <(z-2)/2$ (only for $v\geq 4$). 
We have:
$$
\exists C>0\quad \forall \phi\in\Pc \qquad
\hat{S}^h(\phi) \,\leq\, C
\; .
$$  
\end{prop}
For the last proposition,
we will use technical lemmas,
which will be proved in Subsection~\ref{subsec_technical_lemmas}.

Propositions~\ref{prop_S_hat_S}
and~\ref{prop_hatS}
imply Theorem~\ref{thm_L2_S}.

\subsection{\mathversion{bold}{$S^j$} and \mathversion{bold}{$\hat{S}^j$}}
\label{subsec_S_hat_S}

Here, we prove Proposition~\ref{prop_S_hat_S}.
Let $f\in\Sc(N_v)$.
With the Plancherel formula of Section~\ref{sec_not_tool},
we will prove in this subsection that we have:
\begin{equation}
  \label{equality_nd_Sjf}
\nd{S^j(f)}^2
\,=\,
\int
\sum_{l\in\Nb^{v'}} 
\nn{\hat{S}^j(\phi^{r,\Lambda,l})}^2
\nn{k.\Pi_{r,\Lambda}(f).\zeta_l}^2
dm(r,k,\Lambda)\quad.  
\end{equation}
Because of the Plancherel formula (see Theorem~\ref{thm_plancherel}),
this equality implies:
\begin{eqnarray*} 
\nd{S^j(f)}^2
\,\leq\,
\sup_{\phi\in\Pc}\nn{\hat{S}^j(\phi)}^2
\int
\sum_{l\in\Nb^{v'}} 
\nn{k.\Pi_{r,\Lambda}(f).\zeta_l}^2
dm(r,k,\Lambda)\\
\,=\,
\sup_{\phi\in\Pc}\nn{\hat{S}^j(\phi)}^2\;
\nd{f}^2
\quad. 
\end{eqnarray*}
Thus Proposition~\ref{prop_S_hat_S} is proved,
if equality~(\ref{equality_nd_Sjf}) is satisfied.

Let establish equality~(\ref{equality_nd_Sjf}).
Fubini's equality and the Plancherel formula (Theorem~\ref{thm_plancherel})
for $\partial_s^j(f*\mu_s)\in \Sc(N)$
yield to:
\begin{equation}
  \label{eq:Sj2}
\nd{S^j(f)}^2
=
\int_{s=0}^\infty 
\int
\sum_{l\in\Nb^{v'}} 
\nn{k.\Pi_{r,\Lambda}(\partial_s^j(f*\mu_s)).\zeta_l}^2
dm(r,\Lambda,k)
s^{2j-1}ds
\quad.  
\end{equation}
Because $\mu_s$ is a radial probability measure,
$k.\Pi_{r,\Lambda}(\mu_s)$ 
equals $\Pi_{r,\Lambda}(\mu_s)$,
and we have (Proposition~\ref{prop_pi(f)_rad}):
$\Pi_{r,\Lambda}(\mu_s).\zeta_l=
<\mu_s,\phi^{r,\Lambda,l}>
\zeta_l$.
Then we deduce:
\begin{eqnarray*}
k.\Pi_{r,\Lambda}(f)
k.\Pi_{r,\Lambda}(\mu_s).\zeta_l
&=&
<\mu_s,\phi^{r,\Lambda,l}>
k.\Pi_{r,\Lambda}(f).\zeta_l
\\
k.\Pi_{r,\Lambda}(\partial_s^j(f*\mu_s)).\zeta_l
&=&
\partial_s^j.<\mu_s,\phi^{r,\Lambda,l}>
k.\Pi_{r,\Lambda}(f).\zeta_l
\quad.
\end{eqnarray*}
This last equality, 
with~(\ref{eq:Sj2})
and Fubini's equality,
yield to:
\begin{eqnarray*}
\nd{S^j(f)}^2
=
\int
\sum_{l\in\Nb^{v'}} 
\int_{s=0}^\infty \nn{\partial_s^j.<\mu_s,\phi^{r,\Lambda,l}>}^2s^{2j-1}ds
\nn{k.\Pi_{r,\Lambda}(f).\zeta_l}^2\\
 dm(r,\Lambda,k)\quad.
\end{eqnarray*}
Because of definition of $\hat{S}^j$,
we deduce equality~(\ref{equality_nd_Sjf}).
Thus Proposition~\ref{prop_S_hat_S}
is proved.

\subsection{Boundedness of \mathversion{bold}{$\hat{S}^j$}}
\label{subsec_hatS}

Here, we prove Proposition~\ref{prop_hatS}.
The tools are some computations and properties of special functions
(which are recalled in Subsection~\ref{subsec_specialfcn}).
We remark:
\begin{lem}  
\label{lem_derivee_fcn_s2}
Let $f$ be a smooth function on $\Rb$.
For $h=1,2,\ldots$,
the  derivative  $g^{(h)}(s)$ of  $g(s)=f(s^2)$
is a linear combination of
  $s^{d(j,h)}f^{(h'+j)}(s^2)$, where:
$$
\begin{cases}
      \mbox{if}\; h=2h',\;
  d(j,h)=2j,\; \mbox{on}\; {0\leq j\leq h'}, 
\\
     \mbox{if}\; h=2h'-1,\;
d(j,h)=2j+1,\; \mbox{on}\; {0\leq j\leq h'-1}.\\
\end{cases}
$$
  We remark $d(j,h)+h=2h'+2j$.
\end{lem}
We use also lemmas, 
(whose main lines  of proof are given in the following subsection),
and the following notations:
we fix $\phi=\phi^{r,\Lambda,l}\in\Pc$ and $h=1,2,\ldots$;
we set for $n,\tilde{h}\in\Nb$:
  $$
  \check{b}^{\tilde{h},n}(t,s)
  \,=\,
  \int_{S_1^{(v)}} 
  {(irtx_v)}^{\tilde{h}}
  e^{itsrx_v}
  \partial_{s'}^n
  {\left[
    \overset{v'}{\underset{j=1}  \Pi}
    \flz {l_j} {\lambda_js'r^2\frac{\nn{\pr j X}^2} 2}\right]}_{s'=s^2}
  d\sigma_v(X)
  \quad.
  $$
We denote for $g=(h_1,h_2)\in\Nb^2$, $j=(j_1,j_2)\in\Nb^2$:
$$
I^{g,j}
\,:=\,
\int_0^\infty \nn{b^{g,j}(s)}^2s^{2h-1}ds
\quad,
$$
  \begin{eqnarray*}
    b^{g,j}(s)
    \,:=\,
    s^{d_1+d_2}
    \int_{r=0}^1 
    \check{b}^{\tilde{h},h'_1+j_1}(t,s)\quad 
    \fbs {\frac{z-2}2} ^{(h'_2+j_2)} (s^2\sqrt{1-t^4} \nn{A})\\
    {(\sqrt{1-t^4} \nn{A})}^{h'_2+j_2}
    \;  t^{v-1}{(1-t^4)}^{\frac{z-2}2}dt  
    \quad ,
  \end{eqnarray*}
where $\tilde{h}=h-h_1-h_2$ and for $i=1,2$: 
$d_i:=d(j_i,h_i)$.
If $v$ is odd, we denote $\bar{h}=h_1+h_2$.

\begin{lem}
  \label{lem_maj_partial_h_mu_s_phi}
$\partial_s^h<\mu_s,\phi>$ is a linear combination of $b^{g,j}(s)$
over $g,j$ such that $h_1+h_2+\tilde{h}=h$,
where  $\tilde{h}=0$ if $v$ is even,
and $0\leq j_i\leq h_i/2$, $i=1,2$.
\end{lem}

Proposition~\ref{prop_hatS} will be implied by:

\begin{prop}
\label{prop_maj_I_h}
The integrals
$I^{g,j}$
are bounded independently of $(\Lambda,r,l)\in\Pc$, 
for all parameters 
$g,j$ given in Lemma~\ref{lem_maj_partial_h_mu_s_phi},
as long as $1\leq h<(z-2)/2$ (if $v\geq4$).
\end{prop}

The rest of this subsection is devoted to its proof,
first if $v=2v'+1$ and $\overline{h}\not=0$, 
or if $v=2v'$,
then if $v=2v'+1$ and $\overline{h}=0$.

We will use the following property of the Bessel functions,
deduced from the analyticity of the Bessel function $\fbs \alpha$,
its boundedness on $\Rb$,
and the classical majorations for $J_\alpha$:
\begin{eqnarray}
\forall -1<\beta<2\alpha\quad
\forall n\in\Nb\quad:\quad 
\int_{s=0}^\infty 
\nn{\fbs \alpha ^{(n)} (s)}^2  
s^\beta ds<\infty \quad.
\label{eq_maj_int_fb2}
\end{eqnarray}

\subsubsection{Boundedness of 
  \mathversion{bold}{$I^{(h_1,h_2),(j_1,j_2)}$}
  if  \mathversion{bold}{$\overline{h}\not=0$} 
or \mathversion{bold}{$v=2v'$}}

Here, we assume  ($\overline{h}\not=0,v=2v'+1$) or ($v=2v'$)
and we show Proposition~\ref{prop_maj_I_h}
for these cases as long as $h<(z-2)/2$.
We will use majorations of
$\check{b}^{\tilde{h},n}(t,s)$,
given in the following lemma which will be proved in next subsection.
\begin{lem}
  \label{lem_maj_check_b_h}  
  The expressions $\check{b}^{\tilde{h},n}(t,s)$,
  for $0\leq n\leq \tilde{h}$,
  are bounded up to a constant of 
  $v,\tilde{h}$ by:
  $
  {(\nn{A}t^2)}^n
  s^{-\tilde{h}}
  \sum_{0\leq i\leq \tilde{h}}
  {(\nn{A}s^2t^2)}^i
  $;
if $v=2v'$, we assume $\tilde{h}=i=0$.
\end{lem}

Using this lemma and Cauchy-Schwarz inequality, 
we obtain that $\nn{b^{g,j}(s)}^2$
is then bounded up to a constant by:
  \begin{eqnarray*}
    \nn{A}^{2(h'_2+j_2+h'_1+j_1+i)} 
    s^{2(d_1+d_2)-2\tilde{h}+4i}
    \int_{0}^1
    t^{2(v-1+2i+2(h'_1+j_1))}{(1-t^4)}^{-\frac 12}dt\\
    \int_{0}^1
    \nn{\fbs{\frac{z-2}2}^{(h'_2+j_2)}(s^2\sqrt{1-t^4} \nn{A})
      {(1-t^4)}^{\frac{z-2}2+\frac{h'_2+j_2}2+\frac 14}}^2dt
    \quad.
  \end{eqnarray*}
Finally,  $I^{g,j}$ is bounded up to a constant of $v,h$ by the maximum over
 $0\leq i\leq \tilde{h}$ of:
  \begin{eqnarray*}
    J^{(i)}
    &:=&
    \nn{A}^{2(h'_2+j_2+h'_1+j_1+i)} 
    \int_{0}^\infty
    s^{2(d_1+d_2)-2\tilde{h}+4i}\\
    &&\quad  \int_{0}^1
    \nn{\fbs{\frac{z-2}2}^{(h'_2+j_2)}(s^2\sqrt{1-t^4} \nn{A})
      {(1-t^4)}^{\frac{z-2}2+\frac{h'_2+j_2}2+\frac 14}}^2dt
    \quad s^{2h-1}ds
    \quad.  
  \end{eqnarray*}

We compute another expression of the exponent of $s$ 
(see Lemma~\ref{lem_derivee_fcn_s2})
and the change of variable
  $s'=s^2 \sqrt{1-t^4} \nn{A}$:
  \begin{eqnarray*}
  J^{(i)}
  \,=\,
  \int_{s'=0}^\infty 
  \nn{\fbs{\frac{z-2}2}^{(h'_2+j_2)}(s')}^2
  {s'}^{h'_2+j_2+h'_1+j_1+2i-1}\frac{ds'}2\\
  \int_{0}^1 {(1-t^4)}^{z-2+\frac 12-(h'_1+j_1+2i)}dt  
  \quad.    
  \end{eqnarray*}
Because of (\ref{eq_maj_int_fb2}),
we deduce 
that $J^{(i)}, i=0,\ldots,\tilde{h}$,
$0\leq j_1\leq h'_1$,
$0\leq j_2\leq h'_2$,
are finite as long as $h'_1+h'_2\not=0$ and $h<(z-2)/2$.
Proposition~\ref{prop_maj_I_h}
is thus proved if 
$(\overline{h}\not=0,v=2v'+1)$ or $(v=2v')$.

\subsubsection{Boundedness of
  \mathversion{bold}{$I^{0,0}$}
if  \mathversion{bold}{$v=2v'+1$} and 
\mathversion{bold}{$h=\tilde{h}$}, }

In this sub-subsection,
we will assume $v=2v'+1$ and $h=\tilde{h}<(z-1)/2$, 
so $r\not=0$ and $g=j=(0,0)=0$;
we show Proposition~\ref{prop_maj_I_h} for this case.

We will use the following majoration of
$\nn{\check{b}^{h,0}(r,s)}$:
\begin{lem}
  \label{lem_maj_checkb_h0}
  If $v=2v'+1$,
  $\check{b}^{h,0}(t,s)$ 
  is bounded up to a constant
  by:
  $$
  s^{-h}
  \sum_{0\leq i\leq h}
  {(\nn{A}s^2t^2)}^i
  \quad .
  $$
\end{lem}
Because of the previous majoration of
   $\check{b}^{h,0}(r,s)$ 
  of Lemma~\ref{lem_maj_checkb_h0},
  $\nn{b^{0,0}(s)}^2$ is bounded up to a constant
by the maximum over $i=0,\ldots,h$ of:
\begin{eqnarray*}
&&  {\left(\int_0^1 
   s^{-h}{(\nn{A}s^2t^2)}^i
  \nn{\fb {\frac{z-2}2}{s^2\sqrt{1-t^4} \nn{A}}}
  t^{v-1}{(1-t^4)}^{\frac{z-2}2}dt \right)}^2 \\
&&\quad \leq
    \int_0^1 
\nn{   s^{-h} {(\nn{A}s^2t^2)}^i
  \nn{\fb {\frac{z-2}2}{s^2\sqrt{1-t^4} \nn{A}}}
{(1-t^4)}^{\frac{z-2}2}}^2 dt
    \int_0^1 
  \nn{t^{v-1}}^2dt
\quad,
\end{eqnarray*}
 by H\"older's inequality.
Thus by Fubini's equality,
$I^{0,0}$ is bounded up to a constant by:
\begin{eqnarray*}
&&  \int_0^1 
 \int_0^\infty
  {(\nn{A}s^2t^2)}^{2i}
  \nn{\fb {\frac{z-2}2}{s^2\sqrt{1-t^4} \nn{A}}}^2
\frac{ds}s
{(1-t^4)}^{z-2} dt\\
&&\qquad=\,
 \int_0^\infty
  {s'}^{2i}
  \nn{\fb {\frac{z-2}2}{s'}}^2
\frac{ds'}{2s'}
  \int_0^1 
t^{2i} {(1-t^4)}^{z-2-i} dt
\quad,
\end{eqnarray*}
with Fubini's equality after the change of variable
$s'=s^2\sqrt{1-t^4} \nn{A}$.
Thus by Fubini's equality,
$I^{0,0}$ is finite as long as
$z-2-h>-1$ (because of the $dt$-integral)
and $2h-1<z-2$ (because of the $ds$-integral 
and majoration~(\ref{eq_maj_int_fb2})).

As the other cases where proved in the previous sub-subsection,
Proposition~\ref{prop_maj_I_h}
is also proved as long as
Lemmas~\ref{lem_maj_partial_h_mu_s_phi},
\ref{lem_maj_check_b_h},
\ref{lem_maj_checkb_h0},
 are shown.

\subsection{Proof of the technical lemmas}
\label{subsec_technical_lemmas}

In this subsection, we give the main lines of the proofs 
for  lemmas used in the previous subsection.

\subsubsection{Proof of Lemma~\ref{lem_maj_partial_h_mu_s_phi}}

To prove Lemma~\ref{lem_maj_partial_h_mu_s_phi}, we need the following property of the Bessel functions:
\begin{equation}
\int_{\Sb^n} e^{i<x,y>}d\sigma_n(y)
\,=\,
\fb {\frac{n-2}2} {\nn{x}}
\quad,  \label{fcn_bessel_intsph}
\end{equation}
where $<,>$ denotes the Euclidean scalar product of $\Rb^n$, 
and $\nn{.}$ its norms.

Let $\phi=\phi^{r,\Lambda,l}\in\Pc$.
First, we will need the expression of $<\mu_s,\phi>$.
As the  probability measure $\mu_s$ is radial,
and with the expression of $\phi$ 
given in Subsection~\ref{subsec_specialfcn},
we have:
  \begin{eqnarray*}
    <\mu_s,\phi>
    &=&
    <\mu_s,\Theta^{r,\Lambda,l}>
    \,=\,
    \int_{S_1} \Theta^{r,\Lambda,l} (s.n) d\mu(n)\\
    &=&
    2\int_0^1 \int_{\Sb^v}\int_{\Sb^z}
    \Theta^{r,\Lambda,l}
    (stX,s^2\sqrt{(1-t^4)}A)
    d\sigma_z(A)d\sigma_v(X)\\
    &&\qquad\qquad t^{v-1}{(1-t^4)}^\frac{z-2}2dt
    \quad,
  \end{eqnarray*}
because of the expression~(\ref{expression_mu}) of $\mu$.
We notice:
  \begin{eqnarray*}
\int_{\Sb^z}
    \Theta^{r,\Lambda,l}
    (t_1X,t_2A)
    d\sigma_z(A)
&=&
    e^{i t_1 rx_v}
    \overset{v'}{\underset{j=1}  \Pi}
    \flz {l_j} {\frac{\lambda_j}2 \nn{\pr j {t_1X} }^2}\\
&&\qquad\qquad
\int_{\Sb^z}
    e^{it_2<A,A>}
    d\sigma_z(A)
    \; ;
  \end{eqnarray*}
The last integral against $\sigma_z$ equals
  $\fb {\frac{z-2}2} {t_2 \nn{A}}$,
  because of (\ref{fcn_bessel_intsph}).
We then obtain:
  \begin{eqnarray*}
    <\mu_s,\phi>
    \,=\,
    2\int_0^1
    \int_{\Sb^v}
    e^{itsr x_v}
    \overset{v'}{\underset{j=1}  \Pi}
    \flz {l_j} {\lambda_js^2\frac{\nn{\pr j X }^2} 2}
    d\sigma_v(X)\\
    \fb {\frac{z-2}2} {s^2\sqrt{1-t^4} \nn{A}} \;
    t^{v-1}{(1-t^4)}^{\frac{z-2}2}dt
    \quad .  
  \end{eqnarray*}

We can now show Lemma~\ref{lem_maj_partial_h_mu_s_phi}.
The expression
  $$
  \partial_s^h\left[
    e^{irstx_v}
    \overset{v'}{\underset{j=1}  \Pi}
    \flz {l_j} {\lambda_js^2t^2\frac{\nn{\pr j X}^2} 2}
    \fb {\frac{z-2}2}{s^2\sqrt{1-t^4} \nn{A}}\right]
  \quad,
  $$
is a linear combination over $\tilde{h}+\overline{h}=h$ 
  and $h_1+h_2=\overline{h}$ of:
  $$
  \partial_s^{\tilde{h}}\left[e^{itsrx_v>}\right]
  \partial_s^{h_1}\left[
    \overset{v'}{\underset{j=1}  \Pi}
    \flz {l_j} {\lambda_js^2t^2\frac{\nn{\pr j X}^2} 2}\right]
  \partial_s^{h_2}\left[\fb {\frac{z-2}2}{s^2\sqrt{1-t^4} \nn{A}}\right]
  \quad .
  $$
Computing each derivative 
(using  Lemma~\ref{lem_derivee_fcn_s2} for the second and the third one),
 Lemma~\ref{lem_maj_partial_h_mu_s_phi} is proved.

\subsubsection{Proof of Lemma~\ref{lem_maj_check_b_h} if \mathversion{bold}{$v=2v'$}}

To prove Lemmas~\ref{lem_maj_check_b_h} and~\ref{lem_maj_checkb_h0},
we will use the following remark, 
deduced from the properties of Laguerre functions 
\cite[Lemma 1.5.3]{ThangL}:

\begin{lem}
\label{lem_fcn_lag}
The functions $\flsz k$ and its derivatives are bounded on $\Rb^+$, 
independently of $k$ (but not of the derivation order).
\end{lem}

We assume here $v=2v'$. 
From Lemma~\ref{lem_fcn_lag},
we deduce that
$$
  \partial_{s'}^{n}.
\left[ 
\overset{v'}{\underset{j=1}  \Pi}
 \flz {l_j}  {\lambda_js'r^2\frac{\nn{\pr j X}^2} 2}\right]
  \quad,    
$$
is majorized up to a constant of $n,v$ by the sum of
$$
\nn{A}^n \underset {j\in J}\Pi 
{(t^2\frac{\nn{\pr j X}^2} 2)}^{e_j}
\,=\,
{(t^2 \nn{A})}^n
\underset{j\in J} \Pi {(\frac{\nn{\pr j X}^2} 2)}^{e_j}
\quad,
$$
where the sum runs over the subsets $J$ of $\Nb \cap[1,v']$,
and $e_j\in\Nb, j\in J$,
such that $\sum_{j\in J} e_j=n$.
We thus obtain that 
  $\check{b}^{0,n}(t,s)$
is bounded up to a constant of $n,v$
by ${(t^2 \nn{A})}^n$.
Lemma~\ref{lem_maj_check_b_h} is then proved for $v=2v'$.

\subsubsection{Proof of Lemma~\ref{lem_maj_check_b_h} if
\mathversion{bold}{$\overline{h}\not=0$} and of Lemma~\ref{lem_maj_checkb_h0}}

We assume here $v=2v'+1$ and $\overline{h}\not=0$.
In this case, proof of Lemma~\ref{lem_maj_check_b_h}
is based on suitable $\tilde{h}$ integrations by parts.
We only prove the case $(\tilde{h},n)=(1,0)$.
The general case is similar.

Let choose as atlas of the Euclidean sphere $\Sb^v$ of $\Rb^v$
two caps with poles
  $X_1$ and $-X_1$:
  $$
  C_1
  \,:=\,
  \{ X\in \Sb^v \, ,\; <X,X_1>\;>\; -\frac12\}
  \quad,
  $$
  and
  $$
  C_2
  \,:=\,
  \{ X\in \Sb^v \, , \; <X,-X_1>\;>\; \frac12\}
  \quad.
  $$
  Then we fix a partition of unity, 
  i.e. two smooth functions
  $\psi_1,\psi_2$ on $\Sb^v$ such that:
  $$
  \supp \psi_i \subset C_i
  \quad\mbox{and}\quad
  0\leq \psi_i\leq 1 \;,\; i=1,2\;,
  \quad\mbox{and}\quad
  \psi_1+\psi_2 \,=\,1
  \quad\mbox{on}\;\Sb^v
  \;.
  $$ 
  As chart over $C_1$ (respectively $C_2$),
  we consider the stereo-graphic projection 
  with pole $X_1$ (respectively $-X_1$),
  and we denote $C'_1$ (respectively $C'_2$) its range over $\Rb^{v-1}$.
  $C'_1$ et $C'_2$ are compact subsets of  $\Rb^{v-1}$.
  The points $X=\sum_i x_i X_i$ of the sphere $\Sb^v$
  are parameterized by the coordinates $x_i, i\not=1$;
  furthermore, the measure  $\sigma_v$ 
  is mapped over $C_i, i=1,2$ 
  a measure with a smooth density  $D_i,i=1,2$
  over the Lebesgue measure $dx$.

We decompose the integral $\check{b}^{1,0}(t,s)$
 on this atlas and we integrate by parts.
We obtain:
\begin{eqnarray*}
\nn{\check{b}^{1,0}(t,s)}
&=&  
\left|  \sum_{i=1,2}
    \int_{C'_i} 
    \frac1s e^{isrtx_v}
    \partial_{x_v} \left[\overset{v'}{\underset{j=1} \Pi} 
    \flz {l_j} {\lambda_js^2t^2\frac{\nn{\pr j X}^2} 2}
 x_v \psi_i(X) D_i\right] dx \right|\\
&\leq&
    \frac1s
    \sum_{i=1,2}
    \int_{C'_i} 
    \left| \partial_{x_v} 
      \left[\overset{v'}{\underset{j=1} \Pi} 
    \flz {l_j} {\lambda_js^2t^2\frac{\nn{\pr j X}^2} 2}
        x_v \psi_i(X) D_i\right]\right| dx 
    \quad .
\end{eqnarray*}
($X$ is parameterized by $x$ on each $C_i$.)
From Lemma~\ref{lem_fcn_lag},
and the properties of the chosen atlas,
and with $\nn{A}\sim \sum_{j_0} \lambda_{j_0}$, 
we obtain the majoration of Lemma~\ref{lem_maj_check_b_h}
for $(\tilde{h},n)=(1,0)$.
The rest of the proofs of Lemmas~\ref{lem_maj_check_b_h}
and~\ref{lem_maj_checkb_h0} are similar.



\begin{flushleft}
V\'eronique Fischer\\
Departement of Mathematics, G\"oteborg University,\\
S-412 96  G\"oteborg, Sweden.\\
e-mail: veronfi@math.chalmers.se \hspace{3em} veronique.fischer@math.u-psud.fr\\
\end{flushleft}


\begin{thebibliography}{10}

\bibitem{bourgain}
J.~Bourgain.
\newblock Averages in the plane over convex curves and maximal operators.
\newblock {\em J. Analyse Math.}, 47:69--85, 1986.

\bibitem{coifman_weiss}
Ronald~R. Coifman and Guido Weiss.
\newblock {\em Analyse harmonique non-commutative sur certains espaces
  homog\`enes}.
\newblock Springer-Verlag, 1971.
\newblock \'Etude de certaines int\'egrales singuli\`eres, Lecture Notes in
  Mathematics, Vol. 242.

\bibitem{cowling}
Michael~G. Cowling.
\newblock On {L}ittlewood-{P}aley-{S}tein theory.
\newblock In {\em Proceedings of the Seminar on Harmonic Analysis (Pisa,
  1980)}, number suppl. 1, pages 21--55, 1981.

\bibitem{faraut_h}
Jacques Faraut and Kh{\'e}lifa Harzallah.
\newblock {\em Deux cours d'analyse harmonique}, volume~69 of {\em Progress in
  Mathematics}.
\newblock Birkh\"auser Boston Inc., 1987.
\newblock Papers from the Tunis summer school held in Tunis, August
  27--September 15, 1984.

\bibitem{moi}
V\'eronique Fischer.
\newblock {\em \'Etude de deux classes de groupes nilpotents de pas}.
\newblock PhD thesis, Universit\'e d'Orsay, July 2004.

\bibitem{folland_stein}
G.~B. Folland and Elias~M. Stein.
\newblock {\em Hardy spaces on homogeneous groups}, volume~28 of {\em
  Mathematical Notes}.
\newblock Princeton University Press, 1982.

\bibitem{jacobson}
Nathan Jacobson.
\newblock {\em Lie algebras}.
\newblock Interscience Tracts in Pure and Applied Mathematics, No. 10.
  Interscience Publishers (a division of John Wiley \& Sons), New York-London,
  1962.

\bibitem{margulis_nevo_stein_semisimplegr}
G.~A. Margulis, A.~Nevo, and E.~M. Stein.
\newblock Analogs of {W}iener's ergodic theorems for semisimple {L}ie groups.
  {II}.
\newblock {\em Duke Math. J.}, 103(2):233--259, 2000.

\bibitem{muller_seeger}
Detlef M\"uller and A.~Seeger.
\newblock Singular spherical maximal operators on a class of step two nilpotent
  lie groups.
\newblock {\em to appear Israel Journal}, 2003.

\bibitem{thangavelu2004}
E.~K. Narayanan and S.~Thangavelu.
\newblock An optimal theorem for the spherical maximal operator on the
  {H}eisenberg group.
\newblock {\em Israel J. Math.}, 144:211--219, 2004.

\bibitem{Nevo}
Amos Nevo and Sundaram Thangavelu.
\newblock Pointwise ergodic theorems for radial averages on the {H}eisenberg
  group.
\newblock {\em Adv. Math.}, 127(2):307--334, 1997.

\bibitem{schmidt}
Oliver Schmidt.
\newblock Maximaloperatoren zu hyperfl\"achen in gruppen vom homogenen typ.
\newblock Diplomarbeit an der Christian-Albrechts-Universit\"at zu Kiel, Mai
  1998.

\bibitem{stein_topics_in_an}
Elias~M. Stein.
\newblock {\em Topics in harmonic analysis related to the {L}ittlewood-{P}aley
  theory.}
\newblock Annals of Mathematics Studies, No. 63. Princeton University Press,
  Princeton, N.J., 1970.

\bibitem{stein_Maxfunc}
Elias~M. Stein.
\newblock Maximal functions. {I}. {S}pherical means.
\newblock {\em Proc. Nat. Acad. Sci. U.S.A.}, 73(7):2174--2175, 1976.

\bibitem{stein_Han}
Elias~M. Stein.
\newblock {\em Harmonic analysis: real-variable methods, orthogonality, and
  oscillatory integrals}, volume~43 of {\em Princeton Mathematical Series}.
\newblock Princeton University Press, 1993.

\bibitem{stein_wainger}
Elias~M. Stein and Stephen Wainger.
\newblock Problems in harmonic analysis related to curvature.
\newblock {\em Bull. Amer. Math. Soc.}, 84(6):1239--1295, 1978.

\bibitem{szego}
G{\'a}bor Szeg{\H{o}}.
\newblock {\em Orthogonal polynomials}.
\newblock American Mathematical Society, fourth edition, 1975.
\newblock American Mathematical Society, Colloquium Publications, Vol. XXIII.

\bibitem{ThangL}
Sundaram Thangavelu.
\newblock {\em Lectures on Hermite and Laguerre expansions}, volume~42 of {\em
  Mathematical Notes}.
\newblock Princeton University Press, Princeton, NJ, 1993.
\newblock With a preface by Robert S. Strichartz.

\bibitem{ThangB}
Sundaram Thangavelu.
\newblock {\em Harmonic analysis on the {H}eisenberg group}, volume 159 of {\em
  Progress in Mathematics}.
\newblock Birkh\"auser Boston Inc., 1998.

\bibitem{tit}
E.C. Titchmarsh.
\newblock {\em The theory of functions}.
\newblock London: Oxford University Press, 2nd edition, 1975.

\end{thebibliography}
\end{document}